\documentclass[10pt, reqno]{amsart}

\usepackage{amssymb}
\usepackage{amsthm}
\usepackage{amsmath}
\usepackage[usenames]{color}
\numberwithin{equation}{section}

\title[An inductive Julia-Carath\'eodory theorem]{An inductive Julia-Carath\'eodory theorem for Pick functions in two variables}
\author{
J. E. Pascoe
}
\thanks{Partially supported by National Science Foundation Mathematical
Science Postdoctoral Research Fellowship  
DMS 1606260}


\def\set#1#2{\{ #1 \, | \, #2\}}

\def\ctwo{\mathbb{C}^2}

\def\by-half{b_Y^{-\frac{1}{2}}}

\def\h{\mathcal{H}}

\def\rplus{\mathbb{R}^+}
\def\rtwoplus{{(\mathbb{R}^+)}^2}

\def\be{\begin{equation}}
\def\ee{\end{equation}}

\def\impart#1{{\rm Im}(#1)}

\def\min#1#2{{\rm min}\{#1,#2\}}

\def\cofrac#1#2{\frac{#2}{#1}}

\def\norm#1{\| #1 \|}

\def\be{\begin{equation}}
\def\ee{\end{equation}}
\newcommand{\abs}[1]{\left\vert#1\right\vert}
\def\set#1#2{\{ #1 \, | \, #2\}}

\def\h{\mathcal{H}}

\def\p{\mathcal{P}}
\def\ln{\mathcal{L}^N}
\def\lnminus{\mathcal{L}^{N-}}

\def\lnminusone{\mathcal{L}^{N-1}}

\def\res{{(A-z_Y)}^{-1}}

\def\McCarthy{M\raise.45ex\hbox{c}Carthy }
\def\McCarthyc{M\raise.45ex\hbox{c}Carthy, }

\newcommand\black{\color{black}}

\begin{document}

\bibliographystyle{plain}

\newtheorem{defin}[equation]{Definition}
\newtheorem{lem}[equation]{Lemma}
\newtheorem{prop}[equation]{Proposition}
\newtheorem{thm}[equation]{Theorem}
\newtheorem{claim}[equation]{Claim}
\newtheorem{cor}[equation]{Corollary}
\newtheorem{ques}[equation]{Question}

\begin{abstract}
	We study the asymptotic behavior of Pick functions, analytic functions which take the upper half plane to itself.
	We show that if a two variable Pick function $f$ has real residues to order $2N-1$ at infinity and the
	imaginary part of the remainder between $f$ and this expansion is of order $2N+1,$
	then $f$ has real residues to order $2N$ and directional residues to order $2N+1.$ Furthermore, $f$ has real residues to order
	$2N+1$ if and only if the $2N+1$-th derivative is given by a polynomial,
	thus obtaining a two variable analogue
	of a higher order Julia-Carath\'eodory type theorem.
\end{abstract}
\maketitle
\tableofcontents

\section{Intoduction}
	
	The simplest form of classical Julia-Carath\'eodory theorem, given by Carath\'eodory\cite{car29} and Julia\cite{ju20}, follows.
	\begin{thm}[Julia-Carath\'eodory theorem]
		Let $f:\mathbb{D} \rightarrow \overline{\mathbb{D}}$ be an analytic function.
		The limit
			$\lim_{t\rightarrow 1^-} \frac{1 - |f(t)|}{1-|t|}$
		exists if and only if
			$\lim_{t\rightarrow 1^-} f(t)$
		exists and has modulus $1$
		and
			the directional derivative at $1$ exists for all directions pointing into the disk.
	\end{thm}
	The Julia-Carath\'eodory theorem was extended to higher derivatives by Bolotnikov and Kheifets in \cite{bokh06,bokh08} and earlier, on the upper half plane, by Nevanlinna in his solution of the Hamburger moment problem\cite{nev22}.
There has been some effort to prove an analogue of the Julia-Carath\'eodory theorem in several variables in the works of Abate \cite{ab98,ab05}, Agler \McCarthyc Young \cite{amy11a}, Jafari \cite{jaf93}, and \black Wlodarczyk \cite{wlo87}. We are interested in a fusion of the two approaches, that is, an analogue of
the Julia-Carath\'eodory theorem in several variables concerning higher derivatives.

	Let $\Pi$ denote the upper half plane. On $\Pi$ analogues of the above program exist since $\Pi$ is conformally
	equivalent to $\mathbb{D}.$
	We will give an analogue of the Julia-Carath\'eodory theorem
	on the domain $\Pi^2.$ We work in two variables since operator theoretic representation formulas exist
	for analytic functions $f: \Pi^2 \rightarrow \overline{\Pi},$ but do not exist in general due to some classically
	notorious obstruction\cite{par70, var71}.

	On $\Pi,$ we have the luxury of the Nevanlinna representation.
	\begin{thm}[R. Nevanlinna \cite{nev22}] \label{NevanlinnaRep}
	Let $h: \Pi \to \mathbb{C}$. There exists a finite positive Borel measure $\mu$ on $\mathbb{R}$ such that 
	\begin{equation}\label{NevanlinnaRepFormula}
	h(z) = \int\!\frac{1}{t-z}\, d\mu(t)
	\end{equation}
	if and only if $h$ is analytic, takes values in $\overline{\Pi},$ and
	\begin{equation}\label{NevanlinnaRepFormulaAsy}
	\liminf_{s\to\infty} s\abs{h(is)} < \infty.
	\end{equation}
	Moreover, for any Pick function $h$ satisfying Equation \eqref{NevanlinnaRepFormulaAsy}
	the measure $\mu$ in Equation \eqref{NevanlinnaRepFormula} is uniquely determined.
	\end{thm}
	Notably, the condition \eqref{NevanlinnaRepFormulaAsy} is conformally related to the limit
	in the classical Julia-Carath\'eodory theorem.
	The Nevanlinna representation can be used to develop a theory of higher order regularity, since essentially
	questions about regularity are equivalent to elementary questions in real analysis and measure theory.
	Namely,	since
		$$\frac{1}{t-z} = -\sum^{\infty}_{n=0} \frac{t^n}{z^{n+1}}$$
	questions about regularity at $\infty$ can be reduced to questions about the existence of moments
	$\int t^n d\mu(t).$ The Nevanlinna representation in several variables is given in terms of operator
	theory, and so questions there can be reduced to questions some operator theoretic analogue
	of moments.

\subsection{The L\"owner class}
	We denote the two variable Pick class, the set of holomorphic functions from $\Pi^2$ to $\Pi$, as $\p_2.$
	
	In \cite{am11}, Agler and \McCarthy defined the L\"owner class at infinity.
		\begin{defin} \label{lndfn}
			The  L\"owner class at $\infty,$ denoted $\ln,$ is the set of functions
			$h\in\p_2$ such that $\lim_{s\rightarrow \infty} h(is,is) =0$ and there exists a multi-indexed sequence of real numbers $(\rho_n)_{|n|\leq 2N-1}$
			(here, each $n = (n_1,n_2)$ for some non-negative integers
			$n_1$ and $n_2$ and $|n| = n_1+n_2$)
			such that 
			$$h(z)=\sum_{|n|\leq 2N-1} \frac{\rho_n}{z^n} + o\left(\frac{1}{\|z\|^{2N-1}}\right) \text{ nontangentially}.$$
		\end{defin}
	An asymptotic formula holds \emph{nontangentially} at $\infty$ if for each $c \in \mathbb{R}$
	the formula holds for all $z$ large enough satisfying
		$\norm{z} \le c\ \min{\impart{z_1}}{\impart{z_2}}.$

		A weaker notion of regularity is given by the intermediate L\"owner class.
		\begin{defin}
			The intermediate L\"owner class at $\infty,$ denoted $\lnminus,$ is the set of functions
			$h\in\p_2$ such that $\lim_{s\rightarrow \infty} h(is,is) =0$ and there exists a multi-indexed sequence of real numbers $(\rho_n)_{|n|\leq 2N-2}$
			such that 
				$$h(z)=\sum_{|n|\leq 2N-2} \frac{\rho_n}{z^n} + O\left(\frac{1}{\|z\|^{2N-1}}\right) \text{ nontangentially}.$$
		\end{defin}
	We show that $\ln \neq \lnminus$ in Section $\ref{lnneqlnminus}.$

	We examine an inductive relationship between $\lnminusone, \lnminus,$ and $\ln,$ which is given
	in the following two theorems.
	
	Our first main result describes when a function in $\lnminusone$ is in $\lnminus.$
		\begin{thm} \label{thmoneprime}
			Let $h\in \p_2.$
			The following are equivalent:
			\begin{enumerate}
				\item $h \in \lnminus.$
				\item $h\in \lnminusone$ and for each $b\in\rtwoplus,$
					$$s^{2N-1}\text{Im }\left[h(isb)-\sum_{|n|\leq{2N-3}} \frac{\rho_n}{(isb)^{n}}\right]$$
				is bounded for large $s$.
			\end{enumerate}
		\end{thm}
		We prove Theorem \ref{thmoneprime} as Theorem \ref{thmone} in terms
		of the language
		of the  Agler-\McCarthy vector moment theory, which we will discuss later.

	Our second main result describes when a function in $\lnminus$ is in $\ln.$
		\begin{thm} \label{thmtwoprime}
			Let $h\in \p_2.$
			The following are equivalent:
			\begin{enumerate}
				\item $h \in \ln.$
				\item $h \in \lnminus$ and there are residues, not necessarily real, $\{\rho_n\}_{n\leq 2N-1}$ such that
					$$h(z) = \sum_{|n|\leq 2N-1} \frac{\rho_n}{z^n}+ o\left(\frac{1}{\|z\|^{2N-1}}\right)$$
				nontangentially.
			\end{enumerate}
		\end{thm}
		We prove Theorem \ref{thmtwoprime} as Theorem \ref{thmtwo} in terms
		of the language of the  Agler-\McCarthy vector moment theory.

		For $N=1$ our theorems are conformally equivalent the two-variable
		Julia-Carath\'eodory theorem proven on $\mathbb{D}^2$ in 
		Agler \McCarthyc Young \cite{amy11a} where the conformal
		analogues of $\mathcal{L}^1$
		and $\mathcal{L}^{1-}$ were called C-points and B-points.
		An analysis for $N=1$ was given on $\Pi^2$ in  Agler, Tully-Doyle, Young\cite{aty11,aty12}.

\section{The Agler-\McCarthy vector moment theory}
	A calculus was developed to calculate the residues of functions in $\p_2$ at $\infty$ in \cite{aty11, am11}.
		\begin{thm}[Type I two variable Nevanlinna representation \cite{aty11}]
			Let $h\in \p_2$ and suppose that
				$sh(is,is)$
			is bounded for real $s$ large enough.
			Then, there is a separable \black Hilbert space $\mathcal{H},$ an unbounded self-adjoint operator $A$ on $\mathcal{H},$
			a positive contraction $Y$ and a vector $\alpha \in \mathcal{H}$ such that
				$$\langle(A-z_Y)^{-1}\alpha,\alpha\rangle $$
			where $z_Y = Yz_1 + (1-Y)z_2.$
		\end{thm}
	In terms of the above representation, Agler and \McCarthy defined \emph{vector moments,}
	which occur in a way algebraically analogous
	to the way classical moments occur in the theory of the Nevanlinna representation
	in one variable.\cite{nev22}.
		\begin{defin}
			Given a separable \black Hilbert space $\mathcal{H},$ an unbounded self-adjoint operator $A$ on $\mathcal{H},$
			a positive contraction $Y$ and a vector $\alpha \in \mathcal{H},$
			we  say \black $A$ has \emph{vector moments} to order $N$ denoted $(R_k)^N_{i=1}$ if 
				$$R_k(b)=(b_Y^{-1}A)^{k-1}b_Y^{-1}\alpha$$
			exists for every $b\in \rtwoplus.$
				
			If $R_k$ is a vector-valued polynomial in $\frac{1}{b_1}$ and $\frac{1}{b_2},$ that is, there are vectors $(\alpha_n)_{|n|=k}$ such that
				$$R_k(b) = \sum_{|n|=k} \frac{1}{b^n}\alpha_n ,\black$$
			we extend $R_k$ to all of $\mathbb{C}^2$ via its formula.
		\end{defin}


	
	To prove Theorem \ref{thmoneprime} we prove the following equivalence in terms of the Agler-\McCarthy 
	vector moment theory.
		\begin{thm} \label{thmone}
			Let $h\in \p_2.$
			The following are equivalent:
			\begin{enumerate}
				\item $h \in \lnminus.$
				\item $h \in \lnminusone$ and for any type I representation of $h,$
					$$h(z) = \langle(A-z_Y)^{-1}\alpha,\alpha\rangle ,$$
				$A$ has real vector $(Y,\alpha)$-moments to order $N-1.$
				\item For each $b\in\rtwoplus,$ 
					$$s^{2N-1}\text{Im }\left[h(isb)-\sum_{|n|\leq{2N-3}} \frac{\rho_n}{(isb)^{n}}\right]$$
				is bounded for large $s$.
			\end{enumerate}
		\end{thm}
		We prove Theorem \ref{thmone}	in several parts. The implication
		$(1)\Rightarrow	(2)$ is given in Proposition \ref{prop3.100}. The implication
		$(2)\Rightarrow (1)$ is given in Proposition \ref{prop3.200}. The implication
		$(2) \Leftrightarrow (3)$ is given in Proposition \ref{thm3.400}.
		
		Our second result, Theorem  \ref{thmtwoprime}, becomes the following in the language of the Agler-\McCarthy moment theory.
		\begin{thm} \label{thmtwo}
			Let $h\in \p_2.$
			The following are equivalent:
			\begin{enumerate}
				\item $h \in \ln.$
				\item $h \in \lnminus$ and for any type I representation of $h,$
					$$h(z) = \langle(A-z_Y)^{-1}\alpha,\alpha\rangle ,$$
				$A$ has vector $(Y,\alpha)$-moments to order $N-1$ and $R_{N-1}$ is a vector valued polynomial.
				\item $h \in \lnminus$ and there are residues, not necessarily real, $\{\rho_n\}_{n\leq 2N-1}$ such that
					$$h(z) = \sum_{|n|\leq 2N-1} \frac{\rho_n}{z^n}+ o\left(\frac{1}{\|z\|^{2N-1}}\right)$$
				nontangentially.
			\end{enumerate}
		\end{thm}
		Theorem \ref{thmtwo} is also proven in several parts. $(1)\Leftrightarrow	(2)$
		follows directly from the Agler-\McCarthy moment theory, specifically
		 their theorem given here as Theorem \ref{hvmstores}, in the light of
		Theorem \ref{thmone}. The implication $(1)\Leftrightarrow (3)$ is proven as
		Proposition \ref{thm3.700}.

	\subsection{Some facts about moments}
		
		In \cite{am11}, Agler and \McCarthy proved the following:
		\begin{thm}[Agler, \McCarthy \cite{am11}] \label{hmvsrce}
		Let $\h$ be a Hilbert space, let $\alpha \in \h$ and assume that $A$ and $Y$ are operators acting on $\h$, with $A$ an unbounded self-adjoint and $Y$ a positive contraction. The following conditions are equivalent.\\ \\
		(i)\ \ \ $A$ has finite complex vector $(Y,\alpha)$-moments to order $N-1$ and \\
		\hspace*{10mm} for each $l=1,\ldots,N$ there exist vectors $\alpha_n$, $\abs{n}=l$ such that
		\be
		R_l(z)=\sum_{\mid n\mid=l}\frac{1}{z^n}\alpha_n \notag
		\ee
		\ \ \ \ whenever $z \in \ctwo \setminus \set{z}{z_2 \ne 0, z_1/z_2\notin (-\infty, 0]}$.
		\\\\
		(ii)\ \ \ $A$ has finite real vector $(Y,\alpha)$-moments to order $N-1$ and\\
		\hspace*{10mm} for each $l=1,\ldots,N$ there exist vectors $\alpha_n$, $\abs{n}=l$ such that
		\be\label{eq2.4.6}
		R_l(b)=\sum_{\mid n\mid=l}\frac{1}{b^n}\alpha_n
		\ee
		\ \ \ \ whenever $b \in {\rplus}^2$.
		\end{thm}
		
		We also define scalar moments.
		\begin{defin}
		The $k$th real scalar moment is
			$$r_k(b) = \langle R_{\lceil{k/2}\rceil}(b),AR_{\lfloor{k/2}\rfloor}(b)\rangle . $$
		
		\end{defin}
		Notably, the $r_k$ are always real valued when they are defined and furthermore if the $r_k$
		are given by polynomials in $\frac{1}{b_1}, \frac{1}{b_2}$ then they must have real coefficients.
		
		We will use the following key result about scalar moments.
		\begin{thm}[Agler, \McCarthy \cite{am11}]\label{hvmstores} \label{hmvsmtr}
			A function $h\in \ln$ if and only if $h$ has a type I representation
			$$h(z)= \langle\res \alpha,\alpha\rangle ,$$
			such that $A$ has polynomial vector $(Y,\alpha)$-moments to order $N-1.$
			Moreover,
			$$r_{k}(z)=-\sum_{|n|=k} \frac{\rho_n}{z^n}$$
			where $\rho_n$ are as in Definition \ref{lndfn}.
		\end{thm}
		
		The following telescoping lemma gives a formula that will let us prove the main results.
		\begin{lem}\label{lem3.100}
			Let $h\in\p_2$ with type I representation $$h(z) = \langle\res \alpha,\alpha\rangle ,$$
			be such that $A$ has  vector $(Y,\alpha)$-moments to order $N-1$ and
			scalar moments up to order $2N-1.$
			Let $b\in \rtwoplus.$ Let $X_b = b_Y^{-1/2}Ab_Y^{-1/2}.$ Let $\beta_k = X_b^{k}b_Y^{-1/2}\alpha.$
			Then,
				$$h(isb) + \sum^{2N-1}_{k=1} r_{k}(isb) =
				\cofrac{(is)^{2(N-1)}}{\langle[(X_b - is)^{-1} + (is)^{-1}] \beta_{N-1},\beta_{N-1}\rangle }.$$
		\end{lem}
		\begin{proof}
			Note $b_Y^{-1/2}$ exists since $b_Y$ is strictly positive.
			Note that the expressions for $r_{2k-1}$ and $r_{2k}$ in the notation of the lemma become:
				$$r_{2k-1}(isb) = (is)^{-(2k-1)}\langle\beta_k,\beta_k\rangle $$
				$$r_{2k}(isb) = (is)^{-2k}\langle\beta_{k-1},\beta_k\rangle .$$
			The proof will go by induction.
			When $N=1,$
				$$h(isb) + r_{1}(isb) = \langle(A-isb_Y)^{-1}\alpha,\alpha\rangle  + \langle(isb_Y)^{-1}\alpha,\alpha\rangle $$
				$$= \langle(X_b-is)^{-1}\beta_0,\beta_0\rangle  + \langle(is)^{-1} \beta_0,\beta_0\rangle $$
				$$= \langle[(X_b-is)^{-1} + (is)^{-1}]\beta_0,\beta_0\rangle .$$
			So we are done.
			Now suppose, by induction,
				$$h(isb) + \sum^{2N-1}_{k=1} r_{2k-1}(isb) = \cofrac{(is)^{2(N-1)}}{\langle[(X_b - is)^{-1} + (is)^{-1}] \beta_{N-1},\beta_{N-1}\rangle} $$
			and additionally we have vector $(Y,\alpha)$-moments to order $N$ and scalar $(Y,\alpha)$-moments to order $2N+1.$ 
			So,
			\begin{align*}
				& h(isb) + \sum^{2N+1}_{k=1} r_{k}(isb)\\
				&=
				\cofrac{(is)^{2(N-1)}}{\langle[(X_b - is)^{-1} + (is)^{-1}] \beta_{N-1},\beta_{N-1}\rangle}
				+ \cofrac{(is)^{2N}}{\langle\beta_{N-1},\beta_N\rangle}
				+ \cofrac{(is)^{2N+1}}{\langle\beta_{N},\beta_N\rangle} \\
				& =  \cofrac{(is)^{2(N-1)}}
				{\langle[(X_b - is)^{-1} + (is)^{-1} + (is)^{-2}X_b] \beta_{N-1},\beta_{N-1}\rangle }
				+ \cofrac{(is)^{2N+1}}{\langle\beta_{N},\beta_N\rangle} \\
				& =  \cofrac{(is)^{2(N-1)}}{\langle[(X_b - is)^{-1} + (is)^{-2}(X_b+is)] \beta_{N-1},\beta_{N-1}\rangle }
				 + \cofrac{(is)^{2N+1}}{\langle\beta_{N},\beta_N\rangle} \\
				& =  \cofrac{(is)^{2(N-1)}}{\langle[(X_b - is)^{-1} + (is)^{-2}(X_b+is)(X_b-is)(X_b-is)^{-1}] \beta_{N-1},\beta_{N-1}\rangle }
				 + \cofrac{(is)^{2N+1}}{\langle\beta_{N},\beta_N\rangle} \\
				& =  \cofrac{(is)^{2(N-1)}}
				{\langle[(X_b - is)^{-1} + ((is)^{-2}X_b^2 - 1)(X_b-is)^{-1}] \beta_{N-1},\beta_{N-1}\rangle}
				 + \cofrac{(is)^{2N+1}}{\langle\beta_{N},\beta_N\rangle} \\
				& =  \cofrac{(is)^{2N}}{\langle[X_b^2(X_b-is)^{-1}] \beta_{N-1},\beta_{N-1}\rangle }
				+ \cofrac{(is)^{2N+1}}{\langle\beta_{N},\beta_N\rangle}  \\
				& =  \cofrac{(is)^{2N}}{\langle[(X_b-is)^{-1}] X_b\beta_{N-1},X_b\beta_{N-1}\rangle} 
				+ \cofrac{(is)^{2N+1}}{\langle\beta_{N},\beta_N\rangle } \\
				& =  \cofrac{(is)^{2N}}{\langle[(X_b-is)^{-1} +(is)^{-1}] \beta_{N},X_b\beta_{N}\rangle}.
			\end{align*}
			This concludes the proof.
		\end{proof}
\section{Proofs of results}
	\subsection{Proof of operator theoretic results}
		First we endeavor to prove the equivalence of $(1)$ and $(2)$ in Theorem \ref{thmone}. We separate the proof \black into two parts.
		\begin{prop}\label{prop3.100}
			Let $h\in\p_2$ with type I representation $$h(z) = \langle\res \alpha,\alpha\rangle .$$ If
			$h\in\lnminus$, then $h\in\lnminusone$ and $A$ has vector $(Y,\alpha)$-moments to order $N-1.$ 
		\end{prop}
		\begin{proof}
			Suppose $h\in\lnminus.$ Then, $h\in\lnminusone.$
			We will show $A$ has real vector $(Y,\alpha)$-moments to order $N-1.$
			This is sufficient by Theorem \ref{hmvsrce}.
			By Theorem \ref{hmvsmtr},
				$$h(isb) + \sum^{2N-3}_{k=1} r_{k}(isb)= h(isb) - \sum_{|n|\leq 2N-3} \frac{\rho_n}{(isb)^n},$$
			and $A$ has $(Y,\alpha)$-moments to order $N-2.$ 
			So by Lemma \ref{lem3.100}, adopting its notation,
				$$(is)^{2N-1}\left[
				h(isb) - \sum_{|n|\leq 2N-2} \frac{\rho_n}{(isb)^n}\right] = (is)^3\langle[(X_b - is)^{-1} + (is)^{-1}] \beta_{N-1},\beta_{N-1}\rangle  - is\sum_{|n|= 2N-2} \frac{\rho_n}{b^n}.$$
			Since $h\in \lnminus,$ for some $C> 0$,
				$$\abs{(is)^{2N-1}\left[
				h(isb) - \sum_{|n|\leq 2N-2} \frac{\rho_n}{(isb)^n}\right]}\leq C$$
			So,
				$$\abs{(is)^3\langle[{(X_b-is)}^{-1} - (is)^{-1}] \beta_{N-2},\beta_{N-2}\rangle  - (is)\sum_{|n|=2N-2} \frac{\rho_n}{b^n}} \leq C.$$
			Taking the \black real part preserves this inequality. Thus,
				$$\abs{\text{Re } (is)^3\langle[{(X_b-is)}^{-1} + (is)^{-1}] \beta_{N-2},\beta_{N-2}\rangle } \text{  }\leq C.$$
			Simplifying,
				$$\abs{\text{Re } (is)^2\langle\frac{X_b}{X_b-is} \beta_{N-2},\beta_{N-2}\rangle } \text{  }\leq C$$
				$$\abs{\text{Re } s^2\langle\frac{X_b^2+isX_b}{X_b^2+s^2} \beta_{N-2},\beta_{N-2}\rangle } \text{  }\leq C$$
				$$\left\langle\frac{s^2X_b^2}{X_b^2+s^2} \beta_{N-2},\beta_{N-2}\right\rangle  \text{  }\leq C$$
			By the spectral theorem, there is a measure $\mu$ so that,
				$$\left\langle\frac{s^2X_b^2}{X_b^2+s^2} \beta_{N-2},\beta_{N-2}\right\rangle =\int \frac{s^2x^2}{x^2+s^2} \abs{\beta_{N-2}(x)}^2 d\mu(x).$$
			Note the integrand is monotone increasing in $s,$ so apply monotone convergence theorem to get
				$$\int |x\beta_{N-2}(x)|^2 d\mu(x) = \int |\beta_{N-1}(x)|^2 d\mu(x)$$
			exists and is finite.
			So $X_b\beta_{N-2}\in\text{Dom } X_b.$ That is, $(b_Y^{-1/2}Ab_Y^{-1/2})^{N-1}b_Y^{-1/2}\alpha\in\text{Dom }b_Y^{-1/2}Ab_Y^{-1/2}.$
			So, $(Ab_Y^{-1})^{N-1} \in \text{Dom }A.$
			Thus, $A$ has $(Y,\alpha)$-moments to order $N-1.$
		\end{proof}
		
		The other direction goes as follows.
		\begin{prop} \label{prop3.200}
			Let $h\in\p_2$ with type I representation $$h(z) = \langle\res \alpha,\alpha\rangle .$$ Then, if
			$h\in\lnminusone$ and $A$ has vector $(Y,\alpha)$-moments to order $N-1,$ then 
			$h\in\lnminus.$
		\end{prop}
		\begin{proof}
			Suppose $h\in \lnminusone$ and $h$ has $(Y,\alpha)$-moments to order $N-1.$
			By Theorem \ref{hmvsrce}, \black
				$$z_{Y}^{-1}(Az_Y^{-1})^{N-2}\alpha=R_{N-1}(z) = \sum_{|n|=N-1} \frac{1}{z^n}\alpha_n.\black $$
			Since we have $(Y,\alpha)$-moments to order $N-1,$
				$$(Az_Y^{-1})^{N-1}\alpha=AR_{N-1}(z) = A\sum_{|n|=N-1} \frac{1}{z^n}\alpha_n$$
			is well defined. Note, by linear independence of monomials, each $\alpha_n\in \text{Dom} A.$ Thus,
				$$(Az_Y^{-1})^{N-1}\alpha=AR_{N-1}(z) = \sum_{|n|=N-1} \frac{1}{z^n}A\alpha_n.$$
			So, 
			\begin{align*}
				r_{2N-2}(z) &=  \langle AR_{N-1}(z),R_{N-1}(\overline{z})\rangle  \\
				&=  \sum_{|m|=N-1}\sum_{|n|=N-1} \frac{1}{z^nz^m}\langle A\alpha_n,\alpha_m\rangle =\sum_{|n|=2N-2} \frac{\rho_n}{z^{n}},
			\end{align*}
			where $\rho_{n} = \sum_{n+m=2N-2} \langle A\alpha_n,\alpha_m\rangle .$
			Note that if $b\in \rtwoplus$,
				$$r_{2N-2}(isb)=\frac{1}{(is)^{2N-2}}=\sum_{|n|=2N-2} \frac{\rho_n}{z^{n}}$$
			is real valued. Thus, by linear independence \black of monomials, each $\rho_n$ is real valued. So
				$$h(z)-\sum^{2N-2}_{l=1} r_l(z) = \langle(A-z_Y)^{-1}AR_{N-1}(z),R_{N-1}(\overline{z})\rangle $$
				$$\|z\|^{2N-1}(h(z)-\sum^{2N-2}_{l=1} r_l(z)) = \|z\|^{2N-1}\langle(A-z_Y)^{-1}AR_{N-1}(z),R_{N-1}(\overline{z})\rangle .$$
			Now notice
				$$\|z\|^{2N-1}|(h(z)-\sum^{2N-2}_{l=1} r_l(z))| \leq \|z\|\|(A-z_Y)^{-1}\|\|z\|^{N-1}\|AR_{N-1}(z)\|\|z\|^{N-1}\|R_{N-1}(z)\|$$
			is nontangentially bounded.
			So, $h\in \lnminus.$
		\end{proof}
		This concludes the proof of the equivalence of $(1)$ and $(2)$ in Theorem \ref{thmone}.
	
	\subsection{Proof of function theoretic results}
	We now seek to prove the implication $(1) \Leftrightarrow (3)$ in Theorem \ref{thmone}
	and Theorem \ref{thmtwo}.
	
	We begin with the following lemma which will allow us to prove $(1) \Leftrightarrow (3)$
	for Theorem \ref{thmone}.
		\begin{lem}\label{thm3.500}
			Let $h\in\p_2.$ Suppose $h\in\lnminusone$ and for each $b\in \rtwoplus,$ for large $s,$
				$$(is)^{2N-1}\text{Im }\left[h(isb)-\sum_{|n|\leq{2N-3}} \frac{\rho_n}{(isb)^{n}}\right]$$
			 is bounded.
			Then \black 
				$$r_{2N-1}(b) = \lim_{s\rightarrow\infty}  (is)^{2N-1}\text{Im }\left[h(isb)-\sum_{|n|\leq{2N-3}} \frac{\rho_n}{(isb)^{n}}\right].$$
		\end{lem}
		\begin{proof}
		Suppose 
			$h\in\lnminusone$ and for each $b\in \rtwoplus,$
			$$J_b(s) := s^{2N-1}\text{Im }\left[h(isb)-\sum_{|n|\leq{2N-3}} \frac{\rho_n}{(isb)^{n}}\right]$$
		is bounded.
		Let $h$ have a \black type I representation  $h(z) = \langle\res \alpha,\alpha\rangle .$
		We will show $A$ has vector $(Y,\alpha)$-moments to order $N-1$ and apply the equivalence of 1 and 2 in Theorem \ref{thmone}.
		By Lemma \ref{lem3.100}, adopting its notation,
			$$h(isb)-\sum_{|n|\leq{2N-3}} \frac{\rho_n}{(isb)^{n}} = (is)^{-2(N-2)}\langle[(X_b - is)^{-1} + (is)^{-1}] \beta_{N-2},\beta_{N-2}\rangle .$$
		With this substitution,
			$$J_b(s)
			=(is)^{2N-1}\text{Im }
			(is)^{-2(N-2)}\langle[(X_b - is)^{-1} + (is)^{-1}] \beta_{N-2},\beta_{N-2}\rangle .$$
		Simplified, we obtain
			$$J_b(s)=\langle\frac{-s^2X_b^2}{X_b^2+s^2}\beta_{N-2},\beta_{N-2}\rangle .$$
		Applying the spectral theorem and monotone convergence theorem as in the proof of the equivalence on $(1)$ and $(2)$ in Theorem \ref{thmone}, we get
			$$\lim_{s\rightarrow\infty}  (is)^{2N-1}\text{Im }\left[h(isb)-\sum_{|n|\leq{2N-3}} \frac{\rho_n}{(isb)^{n}}\right] = -\|\beta_{N-1}\|^2= r_{2N-1}(b).$$
		\end{proof}
		
		We now prove $(1) \Leftrightarrow (3)$
	for Theorem \ref{thmone}.
		\begin{prop}\label{thm3.400}
			Let $h\in\p_2.$ Then,
			$h\in\lnminus$ if and only if $h\in\lnminusone$ and for each $b\in \rtwoplus,$ for large $s,$
			$$s^{2N-1}\text{Im }\left[h(isb)-\sum_{|n|\leq{2N-3}} \frac{\rho_n}{(isb)^{n}}\right]$$
			is bounded.
		\end{prop}
		\begin{proof}
			Suppose $h\in\lnminus.$
			The term $\sum_{|n|={2N-2}} \frac{\rho_n}{(isb)^{n}}$ is real. So,
			$$s^{2N-1}\text{Im }\left[h(isb)-\sum_{|n|\leq{2N-3}} \frac{\rho_n}{(isb)^{n}}\right]
			=s^{2N-1}\text{Im }\left[h(isb)-\sum_{|n|\leq{2N-2}} \frac{\rho_n}{(isb)^{n}}\right],$$ 
			which is bounded since $h\in \lnminus.$

			On the other hand, suppose 
			$h\in\lnminusone$ and for each $b\in \rtwoplus,$
			$$s^{2N-1}\text{Im }\left[h(isb)-\sum_{|n|\leq{2N-3}} \frac{\rho_n}{(isb)^{n}}\right]$$
			is bounded.
			By Theorem \ref{thm3.500}, we have scalar moments to order $2N-1$ and thus vector $(Y,\alpha)$-moments.
			So by the equivalence of $(1)$ and $(2)$ in Theorem \ref{thmone}, we are done.
		\end{proof}
	
		The following finishes the proof of Theorem \ref{thmtwo} by showing that 
		We now prove $(1) \Leftrightarrow (3).$
		\begin{prop} \label{thm3.700}
			Let $h\in\p_2.$ Then, $h\in \ln$ if and only if $h \in \lnminus$ and there are residues, not necessarily real, $\{\rho_n\}_{n\leq 2N-1}$ such that
				$$h(z) = \sum_{|n|\leq 2N-1} \frac{\rho_n}{z^n}+ o\left(\frac{1}{\|z\|^{2N-1}}\right)$$
			nontangentially.
		\end{prop}
		\begin{proof}
			The forward direction is true by definition.

			On the other hand, suppose $h \in \lnminus$ and the residues exist.
			Let $b\in \rtwoplus.$ Let $h$ have a type I representation
				$$h(z) = \langle\res \alpha,\alpha\rangle .$$
			By Lemma \ref{lem3.100}, adopting its notation,
				$$h(isb) + \sum^{2N-1}_{k=1} r_{k}(isb) = (is)^{-2(N-1)}\langle[(X_b - is)^{-1} + (is)^{-1}] \beta_{N-1},\beta_{N-1}\rangle .$$
			Multiply through by $(is)^{2N-1}.$
				$$(is)^{2N-1}\left[ h(isb) + \sum^{2N-1}_{k=1} r_{k}(isb)\right] = \langle[is(X_b - is)^{-1} + 1] \beta_{N-1},\beta_{N-1}\rangle .$$
			So,
				$$(is)^{2N-1}\left[ h(isb) + \sum^{2N-1}_{k=1} r_{k}(isb)\right] = \langle[\frac{-s^2}{X_b^2 + s^2} + 1] \beta_{N-1},\beta_{N-1}\rangle .$$
			Applying the spectral theorem and taking limits gives
				$$\lim_{s\rightarrow \infty}(is)^{2N-1}\left[ h(isb) + \sum^{2N-1}_{k=1} r_{k}(isb)\right] = - \|\beta_{N-1}\|^2 + \|\beta_{N-1}\|^2 = 0.$$
			Now 
				$$\lim_{s\rightarrow \infty}(is)^{2N-1}\left[ h(isb) + \sum^{2N-1}_{k=1} r_{k}(isb) - h(isb) + \sum_{|n|\leq 2N-1} \frac{\rho_n}{(isb)^n} \right] = 0.$$
			Applying Theorem \ref{hmvsmtr},
				$$\lim_{s\rightarrow \infty}(is)^{2N-1}\left[ r_{2N-1}(isb) + \sum_{|n|= 2N-1} \frac{\rho_n}{(isb)^n} \right] = 0.$$
			Simplifying,
				$$\lim_{s\rightarrow \infty}r_{2N-1}(b) + \sum_{|n|= 2N-1} \frac{\rho_n}{b^n} = 0,$$
			that is, $$r_{2N-1}(b) = -\sum_{|n|= 2N-1} \frac{\rho_n}{b^n}.$$
		\end{proof}



\section{$\ln\neq\lnminus$} \label{lnneqlnminus}

Now we give an example that shows the hierarchy of L\"owner classes in two variables at infinity does not collapse, that is, $\ln\neq\lnminus,$ which was shown for the case $N=1$ in \cite{amy11a}.
That $\ln\neq\lnminus$ is in stark contrast to the theory in one variable where the
classes are identical\cite{ju20,car29}.

Let
 $\mathcal{H} = l^2(Z_{2(n-1)}),$
and $\pi: Z_{2(n-1)}\rightarrow B(l^2(Z_{2(n-1)}))$ a left regular representation i.e. $\pi(j)e_i = e_{j+i}.$
Let  $A = [\pi(1) + \pi(-1)]$ and $Y$ be a diagonal matrix satisfying
 $Ye_i = e_i$, for $i\neq n$, and $Ye_{n-1} = te_{n-1}.$
Let $\alpha = e_0.$
Let $f$ be the Pick function defined by
	$$f(z) = \langle(A-z_Y)^{-1}\alpha ,\alpha \rangle .$$

Recall
	$R_k(z) = (z_Y)^{-1}(Az_Y^{-1})^{k-1}e_0.$
If $k<n,$ it can be shown inductively that
	$$R_k(z)=z_1^{-k}\sum^{k-1}_{l=0} \begin{pmatrix}
	k-1 \\ l
	\end{pmatrix}e_{-(k-1)+2l}.$$
Furthermore,
	$$AR_{n-1}(z)=z_1^{-(n-1)}\sum^{n-1}_{l=0} \begin{pmatrix}
	n-1 \\ l
	\end{pmatrix}e_{-(n-1)+2l}$$
and
	$$R_n(z)=\frac{1}{tz_1+(1-t)z_2}z_1^{-(n-1)}e_{n-1} + z_1^{-n}\sum^{n-2}_{l=0} \begin{pmatrix}
	n-1 \\ l
	\end{pmatrix}e_{-(n-1)+2l}.$$
So, $r_{2n-1}(z) = \langle R_n(z),AR_{n-1}(\overline{z})\rangle $ is not a polynomial,
but for $k< 2n-1,$ $r_k$ is a polynomial. That is, $f \in \lnminus,$ but $f \notin \ln.$


\bibliography{references}
\end{document}